\documentclass[11pt]{amsart}
\usepackage{amssymb}
\usepackage{amsfonts}
\usepackage{amsmath}
\usepackage[mathscr]{eucal}
\usepackage{enumerate}
\usepackage{mathrsfs}
\usepackage{url}
\usepackage{xcolor}

\newtheorem{theorem}{Theorem}

\newtheorem{corollary}[theorem]{Corollary}
\newtheorem{lemma}[theorem]{Lemma}
\newtheorem{proposition}[theorem]{Proposition}
\theoremstyle{definition}

\newtheorem{definition}[theorem]{Definition}
\newtheorem{example}[theorem]{Example}
\newtheorem{remark}[theorem]{Remark}

\newtheorem{algorithm}[theorem]{Algorithm}

\begin{document}

\title{The arithmetic extensions of a numerical semigroup}

\author{I. Ojeda}
\address{Departamento de Matem\'aticas, Universidad de Extremadura (Spain)}
\email{\small ojedamc@unex.es}

\author{J.C. Rosales}
\address{Departamento de \'Algebra, Universidad de Granada (Spain)}
\email{jrosales@ugr.es}

\thanks{This work was supported by the regional research groups FQM-024 (Junta de Extremadura/FEDER funds) and  FQM-343 (Junta de Andalucia/Feder) and by the projects PGC2018-096446-B-C21 and MTM2017-84890-P.}


\begin{abstract}
In this paper we introduce the notion of extension of a numerical semigroup. We provide a characterization of the numerical semigroups whose extensions are all arithmetic and we give an algorithm for the computation of the whole set of arithmetic extension of a given numerical semigroup. As by-product, new explicit formulas for the Frobenius number and the genus of proportionally modular semigroups are obtained.
\end{abstract}


\maketitle

\section*{Introduction}

Let $\mathbb{N}$ be the set of non-negative integers. A subset $M$ of $\mathbb{N}$ is called a \textbf{submonoid} of $(\mathbb{N},+)$ if $M$ contains $0 \in \mathbb{N}$ and $M$ is closed under the sum in $\mathbb{N}$. A \textbf{numerical semigroup} is a submonoid $S$ of $(\mathbb{N},+)$ such that $\mathbb{N} \setminus S$ is finite.

If $A$ is a non-empty subset of $\mathbb{N}$, then we will write $\langle A \rangle$ for the submonoid of $(\mathbb{N},+)$ generated by $A$, that is, \[\langle A \rangle = \Big\{ \sum_{i=1}^n u_i a_i\ \mid\ n \in \mathbb{N}\setminus\{0\},\ \{a_1, \ldots, a_n\} \subseteq A\ \text{and}\ \{u_1, \ldots, u_n\} \subset \mathbb{N} \Big\}.\] It is a well-known fact that $\langle A \rangle$ is a numerical semigroup if and only if $\gcd(A) = 1$ (see, for instance, \cite[Lemma 2.1]{libro}). 

If $M$ is a submonoid of $(\mathbb{N},+)$ such that $M = \langle A \rangle$ for some $A \subseteq \mathbb{N}$, then we will say that $A$ is \textbf{system of generators} of $M$. Furthermore, if $M \neq \langle B \rangle$ for every $B \subsetneq A$, then we will say that $A$ is a \textbf{minimal system of generators} of $M$. In \cite[Corollary 2.8]{libro}, it is proved that every submonoid of $(\mathbb{N},+)$ has a unique minimal system of generators, which in addition is finite. We will write $\operatorname{msg}(M)$ for the minimal system of generators of $M$.

If $S$ is a numerical semigroup, the elements in $\mathbb{N} \setminus S$ are called \textbf{gaps of $S$}. The cardinality of $\mathbb{N} \setminus S$ is the called \textbf{genus of $S$} and is denoted $\operatorname{g}(S)$. The smallest integer in $S \setminus \{0\}$ is called the \textbf{multiplicity of $S$} and is denoted $\operatorname{m}(S)$. The greatest integer that does not belong to $S$ is called the \textbf{Frobenius number of $S$} and is denoted $\operatorname{F}(S)$. 

Let $S$ and $T$ be numerical semigroups. We say that $T$ is an \textbf{extension} of $S$ if $S \subseteq T$ and we say that $T$ is an \textbf{arithmetic extension} of $S$ if there exist positive integer numbers $d_1, d_2, \ldots, d_n$ such that \[T = \big\{ x \in \mathbb{N}\ \mid\ d_i x \in S,\ i = 1, \ldots, n \big\}.\] Not all the extensions of a numerical semigroup is arithmetic. In Theorem \ref{Th 9}, we characterize the numerical semigroups whose extensions are all arithmetic.

If $S$ is a numerical semigroup and $d$ is a positive integer, then \[\frac{S}d := \big\{x \in \mathbb{N}\ \mid\ d\, x \in S\big\}\] is a numerical semigroup, too (see \cite[Proposition 5.1]{libro}). The numerical semigroup $\frac{S}d$ is called the \textbf{quotient of $S$ by $d$}. The quotients of numerical semigroups by an element have been studied by different authors; see, for instance, \cite{Comm-Alg, Swanson} and more recently \cite{borzi}.

Notice that $\frac{S}d = \mathbb{N}$ if and only if $d \in S$, and observe that $\frac{S}d$ is an arithmetic extension of $S$, and that an extension $T$ of $S$ is arithmetic if and only if there exists $\{d_1, \ldots, d_n\} \subseteq \mathbb{N} \setminus S$ such that \[T = \frac{S}{d_1} \cap \frac{S}{d_2} \cap \ldots \cap \frac{S}{d_n}.\] Indeed, by definition, $x \in T$ if and only if there exist positive integers $d_1, \ldots, d_n$ such that $d_i x \in S,$ for all $i \in \{1, \ldots, n\};$ equivalently, $x \in \frac{S}{d_i}$ for every $i \in \{1, \ldots, n\}.$  

Besides a theoretical interest, a major motivation for the study of arithmetic extensions of numerical semigroups relies in the following two examples.

In \cite{borzi} it is shown that quotients of Arf numerical semigroups (see \cite{arf}) by a positive integer are Arf numerical  semigroups. Since, the intersection of finitely many Arf numerical semigroups is an Arf numerical semigroup (see, for instance, \cite[Proposition 3.22]{libro}), we conclude that arithmetic extensions are stable for the Arf property. Arf numerical semigroups are a relevant family of numerical semigroups as they are semigroup of values of Arf rings (see \cite{lipman}).

In \cite{proportionally} the notion of proportionally modular numerical semigroup (PM-semigroup, for short) was introduced. In \cite{semigroup} the authors characterize the numerical semigroups that can be written as an intersection of finitely many PM-semigroups, these semigroups are called SPM-semigroups. On the other hand, in \cite{toms} Toms introduced a family of numerical semigroups, that we call T-semigrups, which are $K_0-$groups of certain $C^*-$algebras (see \cite[Theorem 1.1]{toms}). In \cite{canadian} it is shown that the T-semigruops are precisely the SPM-semigroups. As a consequence of \cite[Corollary 4.1]{archiv} we have that a numerical semigroup is SPM-semigroup if and only if the exist positive integer numbers $a, d_1, \ldots, d_n$ such that \begin{equation}\label{ecu0} S = \frac{\langle a, a+1 \rangle}{d_1} \cap \ldots \cap  \frac{\langle a, a+1 \rangle}{d_n}.\end{equation} Thus, we conclude that the T-semigroups are the arithmetic extensions of the numerical semigroups generated by two consecutive integers.

One of the main results in this paper is Theorem \ref{Th 12} which determines the Ap\'ery set of the quotient of a numerical semigroup by a positive integer. As a consequence, explicit formulas for the genus and the Frobenius number of PM-semigroups are obtained.

To conclude the introduction we outline how this note is organized: in Section \ref{Sect1} we study the first properties of the set of arithmetic extensions of a numerical semigroup and we characterize the numerical semigroups whose extensions are all arithmetic. In Section \ref{Sect2}, we see that if $S$ is a numerical semigroup and $n$ is an element in $S$ different from zero, then the Ap\'ery set of $\frac{S}d$ with respect $n$ can be easily computed in terms of the Ap\'ery set of $S$ with respect $n$. This fact will provide us nice explicit formulas for the Frobenius number and the genus of the numerical semigroup $\frac{\langle a, a+1\rangle}d$ (see Remark \ref{Rem 13}) and consequently of any T-semigroup by Proposition \ref{Prop 16}. Finally, in the last section, we take advantage of the results in the previous sections to formulate an algorithm for the computation of all the arithmetic extensions of a numerical semigroup. We finish the paper by giving a simple GAP \cite{gap} implementation of our algorithm; this implementation requires the GAP package \texttt{NumericalSgps} \cite{numericalsgps}.

\section[The set of arithmetic extensions]{The set of arithmetic extensions of a numerical semigroup}\label{Sect1}

Let $\Delta$ be a numerical semigroup and let $d_1, d_2, \ldots, d_n$ be positive integer numbers. We write $\Delta(\{d_1, d_2, \ldots, d_n\}) = \{x \in \mathbb{N}\ \mid\ d_i x \in \Delta,\ i = 1, \ldots, n\}$. Clearly, \[\Delta(\{d_1, d_2, \ldots, d_n\}) = \bigcap_{i=1}^n \frac{\Delta}{d_i} \] is an arithmetic extension of $\Delta$. By convention, we assume that $\Delta(\varnothing) = \mathbb{N}$.

In the following, we will denote by $\mathcal{F}(\Delta)$ the set of arithmetic extensions of $\Delta$. The proof of the following result is straightforward.

\begin{proposition}\label{Prop1}
If $\Delta$ is a numerical semigroup, then \[\mathcal{F}(\Delta) = \{\Delta(X)\ \mid\ X \subseteq \mathbb{N}\setminus \Delta\}.\] 
\end{proposition}

\begin{theorem}\label{Th1}
If $\Delta$ is a numerical semigroup, then \[\mathcal{F}(\Delta) = \big\{ \Delta(X) \mid X \subseteq \mathbb{N}\setminus \Delta\ \text{and}\ X = \operatorname{msg}(\langle X \rangle) \big\}.\]
\end{theorem}

\begin{proof}
Let $\mathcal{A} =  \big\{ \Delta(X) \mid X \subseteq \mathbb{N}\setminus \Delta\ \text{and}\ X = \operatorname{msg}(\langle X \rangle) \big\}$. By Proposition \ref{Prop1}, $\mathcal{A} \subseteq \mathcal{F}(\Delta)$. Let us prove the opposite inclusion. To this end, we consider $S \in \mathcal{F}(\Delta)$. By definition, we have that there exists $X = \{x_1, x_2, \ldots, x_n\} \subseteq \mathbb{N} \setminus \Delta$ such that $S = \Delta(X) = \frac{\Delta}{x_1} \cap \frac{\Delta}{x_2} \cap \ldots \cap \frac{\Delta}{x_n}$. Let $Y = \{y_1, y_2, \ldots, y_m\} = \operatorname{msg}(X)$. Since $Y \subseteq X$, we have that \[\Delta(X) \subseteq \Delta(Y).\] On the other hand, if $y \in \Delta(Y)$, then $y y_j \in \Delta$, for every $j = 1, \ldots, m$ and $y x_i \in \Delta$, because $x_i \in \langle y_1, y_2, \ldots, y_m \rangle$, for every $i = 1, \ldots, n$. Thus $y \in \Delta(X)$. Therefore $\Delta(X) = \Delta(Y) \in \mathcal{A}$.
\end{proof}

In general, not all the extensions of numerical semigroup are arithmetic, as we will see later on. One of the aims of this section is to characterize the numerical semigroups $\Delta$ such that $\mathcal{F}(\Delta)$ agrees with the set of extension of $\Delta$.

If $S$ is a numerical semigroup. An element $x \in \mathbb{N} \setminus S$ is said to be a \textbf{fundamental gap} of $S$ if $\{2x,3x\} \subset S$ (see Section 5 of Chapter 3 in \cite{libro} or \cite{fundamental}). The set of fundamental gaps of $S$ is denoted by $\operatorname{FG}(S)$. Observe that $\operatorname{F}(S) \in \operatorname{FG}(S)$, provided that $S \neq \mathbb{N}$.

The proof of the following result is trivial.

\begin{lemma}\label{lemma2}
Let $S$ be a numerical semigroup and let $x$ be a gap of $S$. Then $x \in \operatorname{FG}(S)$ if and only if $\big\{k\,x \mid k \in \mathbb{N} \setminus \{1\} \big\} \subseteq S$.
\end{lemma}

The proof of the next result is straightforward, too.

\begin{lemma}\label{lemma3}
If $S \neq \mathbb{N}$ is a numerical semigroup, then
\begin{enumerate}[(a)]
\item $\frac{S}d = \mathbb{N}$ if and only if $d \in S$.
\item $\frac{S}{d} = \langle 2, 3 \rangle$ if and only if $d \in \operatorname{FG}(S)$. In particular, $\frac{S}{\operatorname{F}(S)} = \langle 2, 3 \rangle,$
\item $\frac{S}2 \cap \frac{S}3 = S \cup \operatorname{FG}(S)$,
\item if $d \in \mathbb{N} \setminus \{0,1\},$ then $S \cup \operatorname{FG}(S) \subseteq \frac{S}d$.
\end{enumerate}
\end{lemma}

Notice that if $S$ is a numerical semigroup, then $\mathbb{N}$, $\frac{S}{\operatorname{F}(S)}$ and $S \cup \operatorname{FG}(S)$ are (non-necessarily different) arithmetic extension of $S$. Furthermore, the first and the third ones are proper extensions of $S$ and the second one is proper if and only if $S \neq \langle 2, 3 \rangle$. 

\begin{example}\label{Ex3}\mbox{}\par
\begin{enumerate}[1.]
\item The only proper extension of $\langle 2,3\rangle$ is $\mathbb{N}$.
\item The proper extensions of $\langle 3,4,5\rangle$ are $\mathbb{N}$ and $\langle 2,3 \rangle = \frac{\langle 3,4,5\rangle}{2}$.
\item The proper extensions of $\langle 2,5 \rangle$ are $\mathbb{N}$ and $\langle 2,3 \rangle = \frac{\langle 2,5 \rangle}{3}$.
\item The proper extensions of $\langle 3,5,7 \rangle$ are $\mathbb{N},$ $\langle 2,3 \rangle = \frac{\langle 3,5,7\rangle}{4}$ and $\langle 3,4,5\rangle = \frac{\langle 3,5,7\rangle}{2}$.
\item The only proper extensions of $\langle 4,5,7 \rangle$ are $\mathbb{N}$, $\langle 2,3 \rangle = \frac{\langle 4,5,7\rangle}{6},\ \langle 3,4,5 \rangle = \frac{\langle 4,5,7\rangle}{3},\ \langle 2,5 \rangle = \frac{\langle 4,5,7\rangle}{2}$ and $\langle 4,5,6,7 \rangle = \frac{\langle 4,5,7\rangle}{2} \cap \frac{\langle 4,5,7\rangle}{3}$.
\end{enumerate}
Notice that we have shown that all the extensions of the following numerical semigroups: $\langle 2,3\rangle,$ $\langle 3,4,5\rangle,$ $\langle 2,5 \rangle,$  $\langle 3,5,7 \rangle$ and $\langle 4,5,7 \rangle$ are arithmetic.
\end{example}

Notice that the set $\mathcal{F}(\Delta)$ is partially ordered by inclusion. Moreover, we have the following:

\begin{proposition}\label{Prop4}
If $\Delta \neq \mathbb{N}$ is a numerical semigroup and $\mathcal{F}(\Delta)$ is the set of arithmetic extensions of $\Delta$, then
\begin{enumerate}[(a)]
\item $\max(\mathcal{F}(\Delta)) = \mathbb{N}$, 
\item $\min(\mathcal{F}(\Delta)) = \Delta$,
\item $\max(\mathcal{F}(\Delta) \setminus \{\mathbb{N}\}) = \langle 2, 3 \rangle,$
\item $\min(\mathcal{F}(\Delta) \setminus \{\Delta\}) = \Delta \cup \operatorname{FG}(\Delta)$.
\end{enumerate}
\end{proposition}

\begin{proof}
(a) Notice that that $\frac{\Delta}d = \mathbb{N}$ for every $d \in \Delta$ and that all numerical semigroups are submonoids of $\mathbb{N}$.

(b) If $\Delta \neq \mathbb{N}$, then $1 \not\in \Delta$ and $\frac{\Delta}1 = \Delta$. Now, it suffices to observe that $\Delta \subset \frac{\Delta}d$, for every $d \in \mathbb{N} \setminus \{0,1\}$, to conclude that every arithmetic extension of $\Delta$ contains $\Delta$.

(c) By part (b) of Lemma \ref{lemma3}, we have that $\langle 2,3\rangle \in \mathcal{F}(\Delta)$. Now, it suffices to observe that if $S \neq \mathbb{N}$ is a numerical semigroup, then $1 \not\in S$ and consequently $S \subseteq \langle 2,3 \rangle$.

(d) By part (c) of Lemma \ref{lemma3} and the remark before Proposition \ref{Prop4}, we have that $\Delta \cup \operatorname{FG}(\Delta) \in \mathcal{F}(\Delta)\setminus\{\Delta\}$. If $S \in \mathcal{F}(\Delta) \setminus \{\Delta\}$, then there exists $\{d_1, \ldots, d_n\} \subseteq (\mathbb{N}\setminus S) \setminus \{1\}$ such that $S = \frac{\Delta}{d_1}\cap \ldots \cap \frac{\Delta}{d_n}$. By part (d) of Lemma \ref{lemma3}, we have that $\Delta \cup \operatorname{FG}(\Delta) \subseteq \frac{\Delta}{d_i}$ for every $i = 1, \ldots, n$. So, $\Delta \cup \operatorname{FG}(\Delta) \subseteq S$.
\end{proof}

The next example shows (among other things) that there exist extensions of a numerical semigroup that are not arithmetic.

\begin{example}\label{Ejem5}
Let $S = \langle 5,7,9 \rangle$. By direct checking, one can see that $\operatorname{FG}(S) = \{6,8,11,13\}$. Clearly, $S \cup \{13\}$ is an extension $S$ and $S \subsetneq S \cup \{13\} \subsetneq S \cup \operatorname{FG}(S)$. Now, by part (d) of Proposition \ref{Prop4}, we have the $S \cup \{13\}$ cannot be an arithmetic extension of $S$. 
\end{example}

Now, we can take advantage of Proposition \ref{Prop4} to characterize the numerical semigroups whose extensions are all arithmetic.

\begin{theorem}\label{Th 9}
The only numerical semigroups whose extensions are all arithmetic are $\mathbb{N},\ \langle 2,3 \rangle,\ \langle 3,4,5 \rangle,\ \langle 2,5 \rangle,\ \langle 3,5,7 \rangle$ and $\langle 4,5,7 \rangle$.
\end{theorem}

Before proving the previous theorem, we need some lemmas.

\begin{lemma}\label{lemma6}
If $S$ is a numerical semigroup and the cardinal of $\operatorname{FG}(S)$ is greater than $1$, then $S$ has non-arithmetic extensions.
\end{lemma}

\begin{proof}
Since $S \cup \{\operatorname{F}(S)\}$ is a proper extension of $S$ and $\{\operatorname{F}(S)\} \subsetneq \operatorname{FG}(S)$, it follows that $S \cup \{\operatorname{F}(S)\} \subsetneq S \cup \operatorname{FG}(S)$. Thus, by part (d) of Proposition \ref{Prop4}, $S \cup \{\operatorname{F}(S)\}$ cannot be an arithmetic extension of $S$
\end{proof}

\begin{lemma}\label{lemma7}
If $S$ is a numerical semigroup and $\operatorname{FG}(S) = \{\operatorname{F}(S)\}$, then the set gaps of $S$ is $\big\{d \in \mathbb{N} \mid d\ \text{divides}\ \operatorname{F}(S) \big\}$.
\end{lemma}

\begin{proof}
Let $d$ be a gap of $S$, that is to say, $d \in \mathbb{N} \setminus S$. Since $\mathbb{N} \setminus S$ is finite, there exists $n \in \mathbb{N} \setminus \{0\}$ such that $n d \in \mathbb{N} \setminus S$, and $\big\{2(nd), 3(nd)\big\} \subset S$. So, $nd$ is a fundamental gap. Then $n d \in \operatorname{FG}(S) = \{\operatorname{F}(S)\}$ by hypothesis, and we conclude that $d$ divides $\operatorname{F}(S)$. Conversely, if $d$ divides $\operatorname{F}(S)$, say $F(S) = n d$, then $d \not\in S$; otherwise, $F(S) = n d \in S$ which is impossible. 
\end{proof}

\begin{lemma}\label{lemma8}
Let $F$ be a positive integer number. Then \[\mathbb{N} \setminus \big\{d \in \mathbb{N} \mid d\ \text{divides}\ F \big\}\] is a numerical semigroup if and only if $F \in \{1,2,3,4,6\}$.
\end{lemma}

\begin{proof}
Since $\big\{d \in \mathbb{N} \mid d\ \text{divides}\ F \big\}$ is a finite set, we have that the set $\mathbb{N} \setminus \big\{d \in \mathbb{N} \mid d\ \text{divides}\ F \big\}$ is a numerical semigroup if and only if given two positive integers $x$ and $y$ not dividing $F$, then their sum does not divides $F$ either. By \cite[Lemma 2]{fundamental}, it is easy to prove this is only possible if, and only if, $F \in \{1,2,3,4,6\}$.
\end{proof}

\noindent\emph{Proof of Theorem \ref{Th 9}.} Let $S \neq \mathbb{N}$ be a numerical semigroup whose extensions are all arithmetic. By Lemma \ref{lemma6}, $\operatorname{FG}(S) = \{\operatorname{F}(S)\}$. Then, by Lemma \ref{lemma7}, \[S = \mathbb{N} \setminus \big\{d \in \mathbb{N} \mid d\ \text{divides}\ \operatorname{F}(S) \big\}.\] Now, since $S$ is a numerical semigroup, by Lemma \ref{lemma8} we have that $\operatorname{F}(S) \in \{1,2,3,4,6\}$. Therefore $S$ is equal to one of the following numerical semigroups $\mathbb{N} \setminus \{1\} = \langle 2,3 \rangle,\ \mathbb{N} \setminus \{1,2\} = \langle 3,4,5 \rangle,\ \mathbb{N} \setminus \{1,3\} = \langle 2, 5 \rangle,\ \mathbb{N} \setminus \{1,2,4\} = \langle 3,5,7 \rangle$ or $\mathbb{N} \setminus \{1,2,3,6\} = \langle 4,5,7 \rangle$.

Finally, since all the extensions of the five above numerical semigroups are arithmetic, as it was shown in Example \ref{Ex3}, we are done.\qed

\section[The Ap\'ery set of the quotient]{The Ap\'ery set of the quotient of a numerical semigroup}\label{Sect2}

If $S$ is a numerical semigroup, then the \textbf{Ap\'ery set of $S$ with respect to $n \in S \setminus \{0\}$} is $\operatorname{Ap}(S,n) := \{x \in S \mid x-n \not\in S\}.$ It is known (see, for instance, \cite[Lemma 2.4]{libro}) that \begin{equation}\label{ecu1}\operatorname{Ap}(S,n) = \{0=\omega(0),\omega(1), \ldots, \omega(w-1)\}\end{equation} where $\omega(i)$ is the least element in $S$ such that $\omega(i) \equiv i\ \text{mod}\  n$, for each $i = 0, \ldots, n-1.$ In particular, we have that the cardinality of $\operatorname{Ap}(S,n)$ is $n$.

\begin{remark}\label{Rem10}
Notice that given $x \in \mathbb{N}$ we have that $x \in S$ if and only if there exists $(k,\omega) \in \mathbb{N} \times \
\operatorname{Ap}(S,n)$ such that $x = k n + \omega$. Therefore, $(\operatorname{Ap}(S,n)\setminus\{0\} ) \cup \{n\}$ is system of generators of $S$.
\end{remark}

The Ap\'ery sets of $S$ provides formulas for the computation of $\operatorname{F}(S)$ and $\operatorname{g}(S)$:

\begin{proposition}\label{Prop11}
If $S$ is a numerical semigroup and $n \in S \setminus \{0\}$, then 
\begin{enumerate}[(a)]
\item $\operatorname{F}(S) = \max \operatorname{Ap}(S,n) - n$.
\item $\operatorname{g}(S) = \frac{1}n \sum_{\omega \in \operatorname{Ap}(S,n)} \omega - \frac{n-1}2.$  
\end{enumerate}
\end{proposition}

\begin{proof}
For a proof, see for instance \cite[Proposition 2.12]{libro}.
\end{proof}

As usual $\lfloor r \rfloor$ and $\lceil r \rceil$ stand for the integer part and the upper integer part of the rational number $r$, respectively. If $a$ and $b$ are integer numbers, then we write $(a\ \text{mod}\  b)$ for the remainder of the Euclidean of $a$ by $b$, in symbols, $(a\ \text{mod}\  b) = a - \lfloor \frac{a}b \rfloor b$.

\begin{theorem}\label{Th 12}
Let $S$ be a numerical semigroup and $n \in S \setminus \{0\}$. If $\operatorname{Ap}(S,n) = \{0=\omega(0),\omega(1), \ldots, \omega(n-1),\}$ and $a \in \mathbb{N} \setminus S$, then 
\[\operatorname{Ap}\left(\frac{S}a,n\right) = \{0, n\, \kappa_1 + 1, n\, \kappa_2 + 2, \ldots, n\, \kappa_{n-1} + n-1\},\] where $\kappa_i = \Big\lceil \frac{\omega\big(
ai\ \text{mod}\  n\big) - a\,i}{a\, n} \Big\rceil,\ i = 1, \ldots, n-1$.
\end{theorem}

\begin{proof}
Let $i \in \{1, \ldots, n-1\}$ and $k \in \mathbb{N}.$ Then $k\, n + i \in \frac{S}a$ if and only if $a(k\, n + i) \in S$. Thus, by Remark \ref{Rem10}, we have that $a\, k\, n + a\, i \in S$ if and only if $a\, k\, n + a\, i  \geq \omega\big(ai\ \text{mod}\  n\big)$. Therefore, $k\, n + i \in \frac{S}a$ if and only if \[k \geq \frac{\omega\big(ai\ \text{mod}\  n\big) - a\,i}{a\, n}.\] Hence, the smallest element in $\frac{S}a$ that is congruent with $i$ modulo $n$ is $\Big\lceil \frac{\omega\big(ai\ \text{mod}\  n\big) - a\,i}{a\, n} \Big\rceil n + i$. Now, by \eqref{ecu1}, we are done. 
\end{proof}

Given a numerical semigroup $S$, we can use the above result to compute the Ap\'ery set of the quotient of $S$ by $b$, and thereafter use Proposition \ref{Prop11}, to compute the Frobenius number and the genus of $\frac{S}b$.  Let us illustrate these computations through an example.

\begin{example}
Let $S = \langle 7,8 \rangle$. Then $\mathrm{Ap}(S,7) = \{\omega(0)=0, \omega(1)=8, \omega(2)=16, \omega(3)=24,  \omega(4)=32, \omega(5)=40, \omega(6)=48\}$. If $\overline S = \frac{S}3$, then, by Theorem \ref{Th 12}, we have that 
\[\mathrm{Ap}\left(\overline S, 7\right) = \{0\} \cup \{7\, \kappa_i + i \mid i = 1, \ldots, 6\},\] where $\kappa_i = \Big\lceil \frac{\omega\big(3i\ \text{mod}\  7)\big) - 3\,i}{21} \Big\rceil,\ i = 1, \ldots, 6$. In this case, $\kappa_1 = 1, \kappa_2 = 2, \kappa_3 = 1, \kappa_4 = 2, \kappa_5 = 0$ and $\kappa_6 = 1$. Therefore, $\mathrm{Ap}\left(\overline S, 7\right) = \{0,8,16,10,18,5,13\}$ and, consequently, \[\operatorname{F}(\overline S) = 18-7 = 11\] and \[\operatorname{g}(\overline S) = \frac{1}7 (8+16+10+18+5+13)-\frac{7-1}2 = 10 - 3 = 7\] by Proposition \ref{Prop11}.
\end{example}

Clearly, if $S$ and $T$ are numerical semigroups, then $S \cap T$ is a numerical semigroup, too. Therefore, from expression \eqref{ecu1} we deduce the following result.

\begin{proposition}\label{Prop 16}
Let $S$ and $T$ be numerical semigroups and $n \in (S \cap T) \setminus \{0\}$. If $\operatorname{Ap}(S,n) = \{0=\omega(0),\omega(1), \ldots, \omega(n-1),\}$ and $\operatorname{Ap}(S',n) = \{0=\omega'(0),\omega'(1), \ldots, \omega'(n-1),\}$, then $\operatorname{Ap}(S \cap T, n) = \{0=\bar \omega(0),\bar \omega(1), \ldots, \bar \omega(n-1),\}$ where $\bar \omega(i) = \max\big(\omega(i), \omega'(i)\big),$ para todo $i = 1, \ldots, n-1$.
\end{proposition}

Thus, combining Theorem \ref{Th 12} and Proposition \ref{Prop 16}, we can compute the Ap\'ery set of all the arithmetic extension of $\Delta$. 

\begin{example}
By direct computation, one can see that the Ap\'ery set of $\langle 4,5,7 \rangle$ with respect to $4$ is $\{0,5,10,7\}$. By Theorem \ref{Th 12}, we have that the Ap\'ery set of $\frac{\langle 4,5,7 \rangle}2 = \langle 2,5 \rangle$ and $\frac{\langle 4,5,7 \rangle}3 = \langle 3,4,5 \rangle$ with respect to $4$ are $\{0,5,2,7\}$ and $\{0,5,6,3\}$, respectively. Therefore, the Ap\'ery set of $\frac{\langle 4,5,7 \rangle}2 \cap \frac{\langle 4,5,7 \rangle}3 = \langle 4,5,6,7 \rangle$ is $\{0,5,6,7\}$.
\end{example}

Using the notation introduced in \cite{proportionally}, a proportionally modular Diophantine inequality is an expression of the form \begin{equation}\label{ecu2} a\, x\ \text{mod}\  b \leq c\, x,\end{equation} where $a,b$ and $c$ are positive integer numbers. Given $a,b$ and $c \in \mathbb{N} \setminus \{0\},$ the set of positive integer solutions of \eqref{ecu2} constitutes a numerical semigroup (\cite[Theorem 13]{proportionally}). These semigroups are called proportionally modular numerical semigroups (PM-semigroups, for short). In \cite[Corollary 3.5]{archiv} the following characterization of PM-semigroup was given.

\begin{proposition}\label{Prop 13}
If $S$ is a PM-semigroup, there exist $a$ and $b \in \mathbb{N} \setminus \{0\}$ such that $S = \frac{\langle a,a+1 \rangle}b$.
\end{proposition}

\begin{remark}\label{Rem 13}
It is a challenging problem to find formulas for the Frobenius number or the genus of $\frac{\langle a,a+1 \rangle}b$ in terms of $a$ and $b$. However, since \[\operatorname{Ap}(\langle a,a+1 \rangle, a) = \{\omega(0)=0, \omega(1)=a+1, \ldots, \omega(a-1) = (a-1)a + (a-1)\},\] we can use Proposition \ref{Prop11} and Theorem \ref{Th 12} to conclude that 
\[
\operatorname{F}\left(\frac{\langle a,a+1 \rangle}b\right) = \max_{i=1, \ldots, a-1} \left(\Big\lceil \frac{(bi\ \text{mod}\ a) (a+1) - b\,i}{a\, b} \Big\rceil a + i\right) - a
\]
and 
\begin{align*}
\operatorname{g}\left(\frac{\langle a,a+1 \rangle}b\right) & = \frac{1}a \left(\sum_{i=1}^{a-1}  \Big\lceil \frac{(bi\ \text{mod}\ a) (a+1) - b\,i}{a\, b} \Big\rceil a + i \right) - \frac{a-1}2 \\ & =\sum_{i=1}^{a-1}  \Big\lceil \frac{(bi\ \text{mod}\ a) (a+1) - b\,i}{a\, b} \Big\rceil.
\end{align*}
Thus, by Proposition \ref{Prop 13}, we have obtained formulas for the Frobenius number and the genus of any PM-semigroup. 

Observe that from Proposition \ref{Prop 16} and Remark \ref{Rem 13} we can deduce a formula for the Frobenius number of any T-semigroup (see expression \eqref{ecu0}).
\end{remark}

\section[An algorithm to compute the arithmetic extensions]{An algorithm to compute the arithmetic extensions of a numerical semigroup}\label{Sect3}

Let $\Delta$ be a numerical semigroup. Our main aim in this section is to give an algorithm for the computation of $\mathcal{F}(\Delta)$. 

\begin{definition}
Let $S$ be a numerical semigroup and let $n \in S \setminus \{0\}$. If $\operatorname{Ap}(S,n) = \{0 = \omega(0), \omega(1) = \kappa_1 n + 1, \ldots, \omega(n-1) = \kappa_{n-1} n + n-1\}$, then the vector \[\varphi_n(S) := (\kappa_1, \ldots, \kappa_{n-1})\] is called the $n-$th \textbf{Kunz-coordinates vector of $S$}.
\end{definition}

The Kunz-coordinates vector of a numerical semigroup where introduced by V.Blanco and J.Puerto in \cite{BP}.  

Observe that Theorem \ref{Th 12} provides a formula for $\varphi_n\left(\frac{\Delta}a\right)$ in terms of $a, n$ and $\operatorname{Ap}(\Delta,n)$.

\medskip
\textbf{Notation.}
Given $\mathbf x = (x_1, x_2, \ldots, x_{n-1})$ and $\mathbf y = (y_1, y_2, \ldots, y_{n-1}) \in \mathbb{N}^{n-1}$, we write $\mathbf x \vee \mathbf y := \big(\max(x_1,y_1), \max(x_2,y_2), \ldots, \max(x_{n-1},y_{n-1}) \big).$

\medskip
As an immediate consequence of Proposition \ref{Prop 16} we have the following result.

\begin{corollary}\label{Cor17}
If $S$ and $T$ are numerical semigroups and $n \in (S \cap T) \setminus \{0\}$, then $\varphi_n(S \cap T) = \varphi_n(S) \vee \varphi_n(T)$.
\end{corollary}

The following result is one of the key point in our algorithm.

\begin{corollary}
If $\Delta \neq \mathbb{N}$ is numerical semigroups and $n \in S \setminus \{0\}$, then set of $n-$th Kunz-coordinates vectors of the arithmetic extensions of $\Delta$ different from $\mathbb{N}$ is \[K:= \left\{\bigvee_{a \in X} \varphi_n\left(\frac{\Delta}a\right) \mid X \subseteq \mathbb{N}\setminus \Delta,\ X \neq \varnothing\ \text{and}\ X = \operatorname{msg}(\langle X \rangle) \right\}\]
\end{corollary}

\begin{proof}
By Theorem \ref{Th1}, the set of arithmetic extensions, $\mathcal{F}(\Delta)$, of $\Delta$ is equal to $\Big\{\displaystyle{\bigcap_{a \in X} \frac{\Delta}a} \mid X\ \text{is a non-empty subset of}\ \mathbb{N}\setminus \Delta\ \text{and}\ X = \operatorname{msg}(\langle X \rangle) \Big\} \cup \mathbb{N}$. Thus, by Corollary \ref{Cor17}, our claim follows.
\end{proof}

Now, we are in condition to give an algorithm for the computation of $\mathcal{F}(\Delta)$.

\begin{algorithm}\label{14}\mbox{}\par
\textsc{Input:} A numerical semigroup $\Delta$.\par
\textsc{Output:} $\mathcal{F}(\Delta).$
\begin{itemize}
\item[1.] Set $\mathcal{F}(\Delta) = \{\mathbb{N},\langle 2,3 \rangle, \Delta\}$ and $S:=\Delta \cup \operatorname{FG}(\Delta)$.
\item[2.] Compute $m:=\operatorname{m}(S)$ and $G:=\mathbb{N} \setminus S$.
\item[3.] For each $a \in G$ compute $\mathbf{v}_a := \varphi_m\left(\frac{\Delta}a\right)$.
\item[4.] Compute the set $\mathcal{X} = \{X \subseteq \mathbb{N}\setminus \Delta,\ X \neq \varnothing\ \text{and}\ X = \operatorname{msg}(\langle X \rangle)\}.$ 
\item[5.] Compute $K := \Big\{\displaystyle{\bigvee_{a \in X} \mathbf{v}_a} \mid X \in \mathcal X\Big\}$.
\item[6.] For each $(\kappa_1, \ldots, \kappa_{m-1}) \in K$, append the numerical semigroup generated by $\{ m, \kappa_1 m + 1, \ldots, \kappa_{m-1} m + m-1\}$ to $\mathcal{F}(\Delta)$.
\end{itemize}
\end{algorithm}

The following GAP \cite{gap} function is a rude implementation of Algorithm \ref{14}. This function requires the GAP package \texttt{NumericalSgps} \cite{numericalsgps}.

{\small
\begin{verbatim}
 ArithmeticExtensions:=function(D)
  local FD,G,S,m,v,a,pow,C,x,sg;

  FD:=[NumericalSemigroup([1]),NumericalSemigroup(2,3), D];
  G:=Gaps(D);
  S:=NumericalSemigroupByGaps(Difference(G,FundamentalGaps(D))); 
 
  m:=Multiplicity(S);
  G:=GapsOfNumericalSemigroup(S);

  v:=[];
  for a in G do
   Append(v,[KunzCoordinates(S/a,m)]); 
  od;
  v:=Set(v);

  pow:=Combinations(v);
  pow:=pow{[2..Length(pow)]};

  C:=[];
  for x in pow do
   x:=TransposedMat(x); 
   Append(C,[List(x,i->Maximum(i))]);
  od;
  C:=Set(C);

  for x in C do
   sg:=Concatenation([m],List(x,i->m*i)+[1..Length(x)]);
   Append(FD,[NumericalSemigroup(sg)]);
  od;
  return Set(FD);
 end;
\end{verbatim}
}

\begin{example}
We can compute the arithmetical extension of $\Delta = \langle 4,6,7 \rangle$ by using the following commands
{\small
\begin{verbatim}
 LoadPackage("NumericalSgps");
 D:=NumericalSemigroup(4,6,7);
 ArithmeticExtensions(D);
\end{verbatim}}
By this way, we obtain that $\Delta$ has four arithmetic extensions, say $\mathbb{N},  \langle 2,3 \rangle, \langle 2,5 \rangle$ and $\langle 4,6,7 \rangle.$ 
\end{example}



\begin{thebibliography}{99}

\bibitem{BP} \textsc{Blanco, V.; Puerto, J.} \emph{An application of integer programming to the decomposition of numerical
semigroups.} SIAM J. Discrete Math. \textbf{26} (2012), no. 3, 1210--1237 

\bibitem{borzi} \textsc{Borz\`i, A.} \emph{A characterization of the Arf property for quadratic quotients of the Rees algebra}. \texttt{arXiv:1806.04448 [math.AC]} 

\bibitem{semigroup} \textsc{Delgado, M.; Garc\'{i}a-S\'anchez, P. A.; Rosales, J. C.; Urbano-Blanco, J. M.} \emph{Systems of proportionally modular Diophantine inequalities}. Semigroup Forum \textbf{76} (2008), no. 3, 469--488. 

\bibitem{numericalsgps} \textsc{Delgado, M.; Garcia-Sanchez, P. A.; Morais, J.} \emph{NumericalSgps, A package  for  numerical semigroups}, Version 1.2.0 dev (2019), (Refereed GAP package), \url{https://gap-packages.github.io/numericalsgps}.

\bibitem{gap}
  The GAP~Group, \emph{GAP -- Groups, Algorithms, and Programming, 
  Version 4.8.6}; 
  2016,
  \url{https://www.gap-system.org}.
  
\bibitem{lipman}
\textsc{Lipman, J.} \emph{Stable ideals and Arf rings} Amer. J. Math., \textbf{93} (1971), 649--685.  

\bibitem{archiv} \textsc{Robles-P\'erez, {A.M.}; Rosales, {J.C.}} \emph{Equivalent proportionally modular Diophantine inequalities}. Arch. Math. (Basel) \textbf{90} (2008), no. 1, 24--30. 

\bibitem{fundamental}
\textsc{Rosales, {J.C.}; Garc\'{\i}a-S\'anchez, {P.A.}, Garc\'{\i}a-Garc\'{\i}a, {J.I.}; Jim\'enez-Madrid, {J.A.}} \emph{Fundamental gaps in numerical semigroups}. J. Pure Appl. Algebra \textbf{189} (2004), 301--313. 

\bibitem{proportionally} \textsc{Rosales, {J.C.}; Garc\'{\i}a-S\'anchez, {P.A.}, Garc\'{\i}a-Garc\'{\i}a, {J.I.}; Urbano-Blanco, {J.M.}} \emph{Proportionally modular Diophantine inequalities}. J. Number Theory \textbf{103} (2003), 281--294.

\bibitem{arf} \textsc{Rosales, {J.C.}; Garc\'{\i}a-S\'anchez, {P.A.}, Garc\'{\i}a-Garc\'{\i}a, {J.I.}; Branco, {M.B.}} \emph{Arf numerical semigroups}. J. Algebra \textbf{276} (2004), 3--12.

\bibitem{Comm-Alg} \textsc{Rosales, {J.C.}; Garc\'{\i}a-S\'anchez, {P.A.}} \emph{Every numerical semigroup is one half of infinitely many symmetric numerical semigroups}. Comm. in Algebra \textbf{36} (2008), 2910--2916.

\bibitem{canadian} 
\textsc{Rosales, J. C.; Garc\'{\i}a-S\'anchez, P. A.} \emph{Numerical semigroups having a Toms decomposition.} Canad. Math. Bull. \textbf{51} (2008), no. 1, 134--139. 

\bibitem{libro} \textsc{Rosales, J.C.; Garc\'{\i}a-S\'{a}nchez, P.A.} \emph{Numerical semigroups}. Developments in Mathematics, vol.\textbf{20}, Springer, New York, (2009).

\bibitem{Swanson} \textsc{Swanson, Irena.} \emph{Every numerical semigroup is one over d of infinitely many symmetric numerical semigroups}. Commutative algebra and its applications, 383--386, Walter de Gruyter, Berlin, 2009.

\bibitem{toms} \textsc{Toms, A.} \emph{Strongly perforated $K_0-$groups of simple $C^*-$algebras}. Canad. Math. Bull. \textbf{46} (2003), no. 3, 457--472. 

\end{thebibliography}
\end{document}